\documentclass{amsart}

\usepackage{amsmath,amsfonts,amssymb}
\usepackage{graphicx,psfrag,subfigure}

\newtheorem{thm}{Theorem}
\newtheorem{rk}{Remark}
\newtheorem{prop}{Proposition}
\newtheorem{clly}{Corollary}

\newtheorem{lemma}{Lemma}
\newtheorem{defi}{Definition}

\newcommand{\R}{{\mathbb{R}}}
\newcommand{\Q}{{\mathbb{Q}}}
\newcommand{\C}{{\mathbb{C}}}
\newcommand{\Z}{{\mathbb{Z}}}
\newcommand{\N}{{\mathbb{N}}}

\newcommand{\Om}{\Omega}
\DeclareMathOperator{\fix}{Fix}
\DeclareMathOperator{\fil}{Fill}
\DeclareMathOperator{\per}{Per}

\begin{document}
\title{Periodic points of covering maps of the annulus}
\author{J.Iglesias, A.Portela, A.Rovella and J.Xavier}

\begin{abstract}
Let $f$ be a  covering map of the open annulus $A= S^1\times
(0,1)$  of degree $d$ , $|d|>1$.  Assume that $f$ preserves an essential (i.e not contained in a disk of $A$) compact
subset $K$. We show that $f$ has at least the same number of periodic points in each period as the map $z^d$ in $S^1.$

\end{abstract}

\maketitle

\section{Introduction.}

Existence of periodic orbits for orientation preserving annulus
homeomorphisms has been extensively studied. One of the motivations is a
celebrated theorem of dynamical systems, the so-called ``last geometric
theorem of Poincar\'e''. Roughly this result says that an area-preserving
homeomorphism of the closed annulus which rotates one boundary component
clockwise and the other counterclockwise possesses at least two fixed
points. This result was conjectured and proved in special cases by
Poincar\'e \cite{poin1}, and was finally proved by Birkhoff \cite{birk}.
This problem has been considered by many authors and was actually the
trigger for a great deal of research (see the paper \cite{DR} for a
historical review on the subject). Since Franks' paper \cite{gpb}, where he
generalized and proved the statement for homeomorphisms of the open annulus,
interest on the problem of existence of periodic orbits for non compact
surface homeomorphisms arose (see, for example, \cite{f2}, \cite{f3}, \cite{fh}, \cite{lc}).

More recently, in \cite{shub} the problem of existence and growth rate of
periodic orbits for degree two surface endomorphism was considered. In this
paper, they deal with a particular case of Problem 3 posed in \cite{shub2}:
let $S$ be the $2$-sphere, and $f:S\to S$ a continuous map of degree $2$; is
the growth rate inequality
\begin{equation*}
\limsup_{n\to \infty}\frac{1}{n}\ln( \#\{\fix(f^n)\}) \geq \ln (2)
\end{equation*}%
\noindent true? The
answer is no, as the map $(r, \theta) \to (2r, 2\theta)$ has only the poles
as periodic points. However, in \cite{shub} it is shown that the growth
inequality holds in a particular case:
if $f$ is $C^1$ and preserves the latitude foliation, then for each $n$, $f^n
$ has at least $2^n$ fixed points.

In this paper, we
study the existence of periodic orbits for covering maps of the open
annulus $f:A\to A$ of degree $d$, $|d|>1$. Note that the growth
inequality holds trivially for the closed annulus $\overline A$ as each connected component of the  boundary of the annulus must
be invariant by $f$ or $f^2$, and we are assuming $|d|>1$.
On the other hand, the covering map $(r, \theta) \to (2r, 2\theta)$ provides a periodic point free example in the open
annulus $\C\backslash\{0\}$. Our result relates both to the theory of annulus homeomorphisms, and to the work in \cite{shub}.

Let us introduce some preliminary definitions.  If $f:A\to A$ is a continuous function, then the homomorphism $f_{*}$ induced by $f$ on the first homology group $H_1(A,\Z)\simeq \Z$, is
$n\mapsto dn$, for some integer $d$. This number $d$ is called the degree of $f$.

We say that an open subset $U\subset A$ is \textit{essential}, if $i_* (H_1
(U, {\mathbb{Z}})) = H_1(A,\Z)=\Z$, where $i_*: H_1 (U, {\mathbb{Z}}) \to
H_1 (A,{\mathbb{Z}})$ is the induced map in homology by the inclusion $i: U
\to A$. We say that a subset $X\subset A$ is \textit{essential} if any
neighbourhood of $X$ in $A$ is essential. We say that a subset is \textit{%
inessential} if it is not essential, or equivalently, if it is contained in
a disk of $A$.  If $x$ is a periodic point for $f$, its {\it period} is the number $\min\{n\geq 1: f^n (x) = x\}$.  We write $\per_n (f)$ for the set of periodic points of period $n$ of a
given map $f$, and $\fix(f)= \per_1 (f)$.

Let $A^*$ be the compactification of the annulus $A$ with two points so that it
is homeomorphic to the two-sphere. Each connected component of $A^*\backslash A
$ is called an \textit{end} of $A$.  Note that if $f$ is a proper mapping, then $f$ extends continuously to $A^*$, and either fixes both ends or interchanges them.

We need one last definition:  what it means for $f$ to be complete;  we postpone this until Section \ref{niel}  because the definition involves some Nielsen theory.  We recall that
$x,y\in \fix(f)$ are {\it Nielsen equivalent} if there exist an arc $\gamma$ joining $x$ and $y$ such that $\gamma$ is homotopic to $f(\gamma)$ with fixed endpoints.  If $f$ is complete then for each $n$, $f^n$ has exactly $|d^n -1|$ Nielsen classes of fixed points, (see Lemma \ref{crcomp} in Section 3).

We prove the following:

\begin{thm}
\label{per} Let $f: A \to A$ be a  covering map of degree $d$, $|d|>1$. Suppose there
exists an  essential continuum $K\subset A$ such that $f(K)\subset K$.
Then $f$ is complete.

\end{thm}

As explained above, the growth rate inequality holds for annulus maps under the standing hypothesis.

Note  that this result is strictly a consequence of degree; an irrational rotation in the open annulus has no periodic points, and has every essential circle as a compact invariant subset. Results in the same line of work have been obtained by Boronski in \cite{bolonqui1} and \cite{bolonqui2}.

The periodic points given by Theorem \ref{per} do not necesarilly belong to $K$ (see Section \ref{ex}, example \ref{e1}). The problem of whether or not the fixed points of a given map with a compact invariant set $K$ belong to $K$ is known in the literature
as  Cartwright-Littlewood theory. In a seminal paper, M. Cartwright and J. Littlewood \cite{cl} proved that if $K$ is a nonseparating continuum of the plane, invariant under an
orientation preserving
homeomorphism $h$, then $h$ has a fixed point {\it in $K$}.  Existence of a fixed point under such hypothesis was already known on account of Brouwer's plane translation theorem
(\cite{brou}), the novelty being
that the fixed point must belong to the set $K$. An easy proof or Cartwright-Littlewood's theorem can be found in  the extraordinary single-page paper of M.Brown \cite{brown}. In \cite{bell},
H. Bell
proved Cartwright-Littlewood's theorem for orientation reversing plane homeomorphisms, and his results were later generalized by K. Kuperberg  in \cite{krys}  for arbitrary plane continua (not necesarilly
nonseparating).
In \ref{e1}, Section \ref{ex}, we construct a degree two covering map $f$ of the annulus with a totally invariant essential continuum $K$  ($f^{-1}(K) = K$), such that $\fix (f)\cap K = \emptyset$.
However, $K$ is not {\it filled}, that is, its complement has bounded connected components . If
$K\subset A$ is any compact set, the set $\fil (K)$ is defined as the union of $K$ with the bounded connected components of its complement. The definition of Nielsen classes of periodic points is
contained in Section \ref{niel}. We prove the following:

\begin{thm}\label{relleno} Let $f: A \to A$ be a  covering map of degree $d$, $|d|>1$. Suppose there
exists an  essential continuum $K\subset A$ such that $f(K)\subset K$.  Then, there exists a representative $x\in \fil(K)$ for each Nielsen class of periodic points of $f$.

\end{thm}

{\bf Notations.} Throughout this article, $A= S^1\times (0,1)$ is the open annulus, $\tilde A=%
{\mathbb{R}}\times (0,1)$ its universal covering space, and $\pi: \tilde
A\to A$ the universal covering projection.
We will denote by $F$ any lift of $f: A \to A$ to the universal covering, that is, $F$ is a map
satisfying $f \pi=\pi F$. Note that  $F (x+1, y) = F (x,y) + (d,0)$ if $f$ has degree $d$. To
lighten notation, if $z\in \tilde A$, we write $z+k$ for the point $z+
(k,0), k\in {\mathbb{Z}}$. The map  $m_d: S^1 \to S^1$ is defined as $m_d (z) = z^d$.

We aknowledge the referee for many useful observations and providing the proof of the example given in 5.7.

\section{Nielsen theory background.}\label{niel}

In this section, we gather the necessary information on Nielsen Theory.  For what follows, $f:A\to A$ is any continuous map. If $p,q\in \fix(f)$, then $p$ and $q$ are said to be {\it Nielsen equivalent} if there exists a curve $\gamma$ from $p$ to $q$ such that $f(\gamma)$ and $\gamma$ are homotopic
with fixed endpoints.
If $p$ and $q$ are periodic points of $f$, then $p$ and $q$ are Nielsen equivalent if they are equivalent as fixed points of some $f^k$, $k\geq 1$.
The definition of Nielsen equivalence does not depend on the choice of $k$ :

\begin{lemma}
\label{lomismo}  Let $p,q \in \fix(f)$ and let $\gamma$ be a curve from $p$ to $q$. If $\gamma\sim f^k\gamma$ for some $k>1$, then $\gamma\sim f\gamma$.
\end{lemma}
\begin{proof} Let $\tilde p$ be a lift of $p$ and let $F$ be the lift of $f$ that fixes $\tilde p$.  Let $\tilde \gamma$ be the lift of $\gamma$ starting at $\tilde p$, and let $\tilde q$ be
the endpoint of $\tilde \gamma$. Assume  that $\gamma$ is not homotopic to $f(\gamma)$.  This implies that $F(\tilde q)\neq \tilde q$, and so there exists $l\in\Z$,
$l\neq 0$ such that $F(\tilde q)= \tilde q+l$.  So, $F^k (\tilde q) = \tilde q+\sum _{i=0}^{k-1} d^i l\neq \tilde q $ and so $f^k\gamma$ is not homotopic to $\gamma$.
\end{proof}

\begin{lemma}
\label{l1}
Let $p$ and $q$ be fixed points of $f$. The following conditions are equivalent:
\begin{enumerate}
\item
$p$ and $q$ are Nielsen equivalent.
\item

If $F$ is any lift of $f$, and $\tilde p$ is a lift of $p$, there exist $\tilde q$ a lift of  $q$  such that $F(\tilde p)- \tilde p= F(\tilde q)-\tilde q$.
\end{enumerate}

\end{lemma}
\begin{proof} (1) $\Rightarrow$ (2): Let $F$ be a lift of $f$, and  $\tilde p$ any lift of $p$.  Then, as $p\in \fix (f)$, there exists $l\in \Z$ such that $F(\tilde p) = \tilde p+l$.  Let $\tilde q$ be the
endpoint of $\tilde \gamma$, the lift of $\gamma$ starting at $\tilde p$, where $\gamma$ is the arc given by the Nielsen equivalence.  As $\gamma \sim f(\gamma)$, the lift of $f(\gamma)$ starting
at $\tilde p +l$
must end at $\tilde q +l$.  On the other hand, this lift must coincide with $F (\tilde \gamma)$, which gives $F(\tilde q) = \tilde q + l$.

(2) $\Rightarrow$ (1):  Let $\tilde p$ be a lift of $p$, and let $F$ be the lift of $f$ such that $F(\tilde p) = \tilde p$. There exist $\tilde q$ a lift of  $q$  such that
$F(\tilde q)= \tilde q$.  Take any arc $\tilde \gamma$ joining $\tilde p$ and $\tilde q$.  Then, $F(\tilde \gamma)$ is obviously homotopic to $\tilde \gamma$.  So, $\gamma = \pi (\tilde \gamma)$
realizes the Nielsen equivalence between $p$ and $q$.

\end{proof}

Note that the number of Nielsen classes of fixed points for the map $m_d$ coincides with its number of fixed points which is $|d-1|$. We will give a simple proof of the following fact:

\begin{thm}
\label{t1}
If $f$ is a map of degree $d$, $|d|>1$, of the annulus, then the number of equivalence classes of fixed points of $f$ is less than or equal to
$|d-1|$.
\end{thm}

\begin{proof}
Let $f$ be a degree $d$ map of the annulus, let $F$ be a lift of $f$ and $p\in \fix(f)$. If $\tilde p$ is any lift of $p$ then there exists an integer $l$ such
that $F(\tilde p )=\tilde p +l$. Moreover, as
$F(x+1)=F(x)+d$, one can choose  $\tilde p$ so that $l$ is an integer between $1$ and $|d-1|$. To see this, note that if $j\in \Z$, then $F(\tilde p +j) = F (\tilde p) + dj = \tilde p +l+dj = \tilde p +j + j (d-1)+l$.  So, there exists a unique $j\in\Z$ such that $F (\tilde p +j) = \tilde p +j +L$, $1\leq L\leq |d-1|$.

This number
$L= \ell_p$ is uniquely determined by $p$ and the lift $F$. The previous lemma implies that $p$ and $q$ are equivalent iff $\ell_p=\ell_q$, which clearly implies the assertion.

\end{proof}

The previous result also follows from the fact that the number of all classes is an homotopy type invariant and that Theorem \ref{t1} holds for the circle (see \cite{bol} pages 618 and 630).

The period of a Nielsen equivalence class is defined as the minimum of the periods of
periodic points of $f$ in the class. Let $N_k(f)$ be the number of different Nielsen classes of period $k$ of the map $f$.

\begin{defi}
A degree $d$ map $f$ of the annulus is said to be {\it complete} if for each positive integer $k$, it holds that $N_{k}(f)=N_k(m_d)$.

\end{defi}

Note that Theorem \ref{t1} implies that a complete map has the maximum possible number of Nielsen classes in each period.

The following two results will be used in the proof of Theorem \ref{per}.

 \begin{lemma}\label{crcomp} A map $f$ is complete iff $N_1(f^k)=N_{1}(m_d^k)$ for every positive $k$.
 \end{lemma}

 \begin{proof} It is clear that $f$ complete implies that condition. For the converse, we proceed by induction to prove that $N_{k}(f)=N_k(m_d)$ for all $k$:
for $k=1$ it is obvious by hypothesis.

Now assume that for every $k'<k$ it holds that $N_{k'}(f)=N_{k'}(m_d)$. Then
$$
   |d^{k}-1| = N_1(f^k)=    N_k(f)+\sum_{k'|k}N_{k'}(f),
$$
where $k'|k$ means $k$ is a multiple of $k'$ and $k'<k$. On the other hand, the same equation holds for $m_d$ in place of $f$, and using the induction hypothesis, the equation $N_k(f)=N_k(m_d)$ follows.

 \end{proof}

\begin{lemma}
\label{d-1} Suppose that $\fix(F)\neq\emptyset$ for every lift $F:\tilde A\to \tilde A$ of $f$. Then $f$ has $|d-1|$ different Nielsen classes of fixed points.
\end{lemma}

\begin{proof}
Fix a lift $F_0=F$ of $f$, and for $k\in {\mathbb{Z}}$ define $F_k (x) = F (x) +
k $. Note that every lift of $f$ belongs to the family $(F_k)_{k\in {\mathbb{%
Z}}}$. For every $k\in {\mathbb{Z}}$, let $x_k\in \tilde A$ be such that $F_k (x_k) = x_k$. We want to show that there are $|d-1|$ different Nielsen classes of fixed points.

Suppose there exists $i\neq 0$ such that $\pi(x_i)$ is Nielsen equivalent to $\pi(x_0)$.  As $F(x_0)=x_0$, then by Lemma \ref{l1}, there exists $l\in \Z$ such that $F(x_i + l ) = x_i +l$.

 As $F_i(x_i)=x_i$ and $F_i(x_i)=F(x_i)+i$ then $F(x_i) = x_i -i$.  So,
 $$ x_i +l=F(x_i + l )=F(x_i)+ld=x_i-i+ld  $$

  Then $i= l (d-1)$.  As $i\neq 0$,then $l\neq 0$ and the points $\pi (x_0), \pi (x_1), \ldots , \pi (x_{|d-1|-1})$ are
all in different Nielsen classes. So, there are at least $|d-1|$ different Nielsen classes, and by Theorem \ref{t1} there are exactly $|d-1|$ different Nielsen classes.
\end{proof}

In most cases we will prove completeness by means of the following corollary:

\begin{clly}\label{final} If for every $k$ every lift of $f^k$ has a fixed point, then $f$ is complete.
\end{clly}

\noindent
{\bf Remarks.}
\begin{enumerate}
\item
Note that the fact that a lift of $f$ has fixed points does not imply that
\textit{every} lift does (see Example \ref{e3}).
\item
The definition of completeness can also be applied to circle maps. If $f$ is a degree $d$ map of the circle, then no matter if it is a covering or not, the number of Nielsen classes of fixed points of $f$ is $|d-1|$. In general, $N_k(f)=N_k(m_d)$ for every $k$. It follows that every circle map is complete.
\item
Theorem \ref{t1} implies that $N_1(f)\leq |d-1|$, and obviously $N_1(f^k)\leq |d^k-1|$. However, it is not true in general that $N_j(f)$ must be less than or equal to $N_j(m_d)$ as the
example in Figure \ref{ejniel} shows.  It is a map of degree $-2$ that has two fixed points $r_1$ and $r_2$, the rays $S_1$ and $S_2$ verify $f(S_1)=S_2$, $f(S_2)=S_1$ and $\{p,q\}$ is a
two periodic  cycle.  The open invariant region bounded by $S_1$ and $S_2$ contains no fixed point.
Note that it has two fixed points (against three of
$m_{-2}$) and has one two periodic cycle formed by $p$ and $q$ (against zero of $m_{-2}$).

It holds that $f^2$ has exactly four fixed points, but it has just three Nielsen classes, one of them contains the two periodic cycle.  According to our definition, this map is not complete and
consequently it cannot leave invariant an essential compact set.

\end{enumerate}

\begin{figure}
\caption{}
\label{ejniel}
\begin{center}
\psfrag{a}{$a$} \psfrag{r1}{$r_1$} \psfrag{r2}{$r_2$}
\psfrag{s1}{$S_1$} \psfrag{s2}{$S_2$}
\psfrag{p}{$p$}\psfrag{q}{$q$}
\includegraphics[scale=0.2]{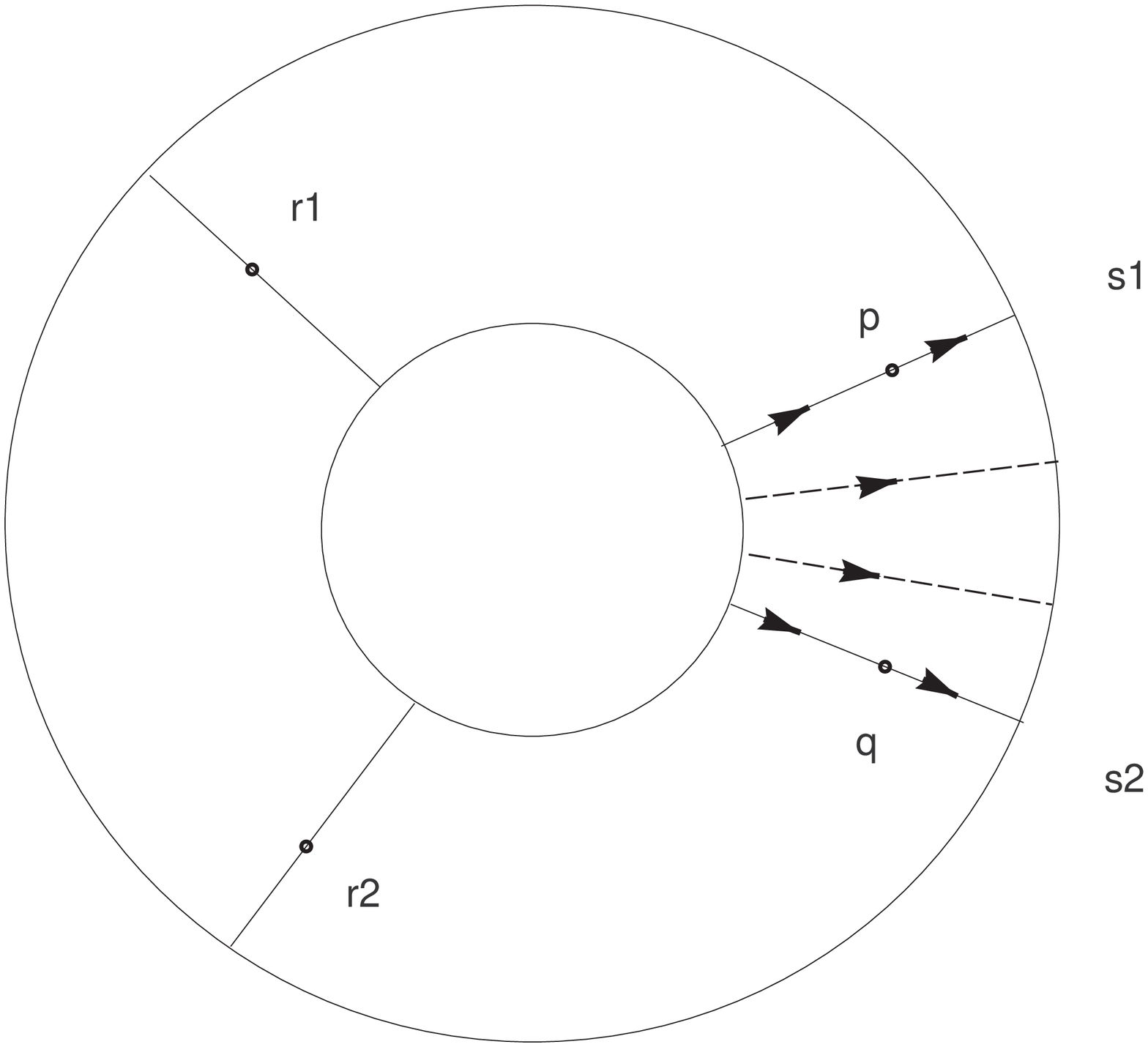}
\end{center}
\end{figure}

\section{Proof of Completeness}\label{proofs}

In this section we prove Theorem \ref{per}. To find periodic orbits we proceed in the standard fashion:  we lift to the universal covering space ${\mathbb{R}}\times (0,1)\sim \R ^2$ and try to use some fixed-point theorems for self-maps of the plane. If $f$
happens to be an orientation preserving covering map, then the lift $F: \R ^2 \to \R ^2$ is an orientation preserving plane homeomorphism, and existence of fixed points is guaranteed by any kind of recurrence:

\begin{thm}
\cite{brou}\label{brouwer} If $F:{\mathbb{R}} ^2 \to {\mathbb{R}} ^2$ is a
fixed point free orientation preserving homeomorphism, then every point is
wandering.
\end{thm}

For a modern exposition of this theorem in its maximum expression, see \cite{patrice}.

Although this technique is quite useful in the case that the map $f$ is isotopic to the identity, recurrence in the lift for maps of degree $d$, $|d|>1$ is not so easy to get.  Indeed, the
lifted map $F$ satisfies $F (x+1, y) = F (x,y) + d$ and so every point wants to escape to infinity exponentially fast.  Of course we may impose some strong hypothesis implying immediately
recurrence for the lift:

\begin{lemma}\label{sintasa}  If $f: A\to A$ is an orientation preserving covering map of degree $d\neq 0$ preserving an inessential continuum $K\subset A$, then $\fix (f) \neq\emptyset$.
\end{lemma}

A continuum is {\it inessential} if it is contained in a disk of $A$.  The proof is immediate from Brouwer's Theorem \ref{brouwer}, as the hypothesis implies that there exists a lift of $f$ that
preserves a compact subset of the plane (namely, a connected component of the preimage of $K$ by the covering projection).

However, if $K$ is essential, no connected component of its lift to the
universal covering space is compact. The proof of Theorem \ref{per} in the orientation preserving case is based on a simple (though key) observation that was already made in \cite{iprx}. If
$f: A \to A$ is a continuous map of degree $d$, $|d|>1$ and $K\subset A$ is a
compact set such that $f(K)\subset K$, then $f|_K$ is semiconjugate to the restriction of $m_d$ to an invariant subset. Existence and properties of the
semiconjugacy $h: K \to S^1$ are contained in Lemma \ref{semi}.  Using Brouwer's Theorem, existence of fixed points is proved in Lemma \ref{fijo} if $h^{-1} (1)\neq \emptyset$, as this
guarantees existence of a compact invariant set for the lift. We prove completeness of $f$ using standard Nielsen theory  if $K$ is essential (note  Lemma \ref{sintasa} only gives a fixed
point, not completeness); this is done in Lemmas \ref{d-1} and \ref{sur}.

If $f$ reverses orientation, we use Kuperberg's theorem in \cite{krys} to find fixed points for orientation reversing plane homeomorphisms.

\begin{thm}\label{kryst}\cite{krys} Let $f$ be an orientation reversing homeomorphism of the plane, and $X$ a continuum of the plane invariant under $f$.  Then, $f$ has at least one fixed point
in $X$.

\end{thm}

The following lemma, which is esentially the Shadowing Lemma for expanding maps is key for the purposes of this paper. See lemma 1 and lemma 2 in \cite{iprx}.

\begin{lemma}
\label{semi} Let $f:A\to A$ be a continuous map of degree $d$, $|d|>1$, and let $
F:\tilde A\to \tilde A$ be a lift of $f$ . Let $K$ be a compact $f$%
-invariant ($f(K)\subset K$) subset of the annulus, and $\tilde K=\pi^{-1}(K)
$. Then there exists a continuous map $H_F:\tilde K\to {\mathbb{R}}$ such
that:

\begin{enumerate}
\item \label{s1}$H_F(x+1,y)=H_F(x,y)+1$,

\item $H_FF=dH_F$,

\item $|H_F(x,y)-x|$ is bounded on $\tilde K$,

\item $H_F(x)=\lim_{n\to \infty} \frac{(F^{n}(x))_1}{d^{n}}$, where $()_1$ denotes projection over the first coordinate.

\end{enumerate}
\end{lemma}

The function $H_F$ appears as a fixed point of the contracting operator $H\to \frac{1}{d}HF$ acting on the space of continuous functions $H:\tilde K\to\R$ that satisfy item (1).

Let $h$ be the quotient function of $H_F$. It is well defined because of item (\ref{s1}) in the previous lemma.

The previous lemma gives:

\begin{clly}
\label{periodicos} Let $f:A\to A$ be a continuous map of degree $d$, $|d|>1$, and  $K\subset A$ be  compact and forward invariant. Then the function $h:A\to S^1$, projection of $H_F$, is a semiconjugacy from the restriction of
$f$ to $K$ to the restriction of $m_d$ to an invariant subset.
\end{clly}

\begin{rk}  Note that we have not yet assumed that $f$ is a covering and thus Lemma \ref{semi} and Corollary \ref{periodicos} are valid for continuous maps of degree $d$, $|d|>1$.
\end{rk}

There is still one preliminary result needed in the proof of the theorem.

\begin{lemma}
\label{extend}
Let $g$ be a covering map of the open annulus $A$, and $K$ a compact subset of $A$. Then there exists a covering $g'$ of the closed annulus such that $g'=g$ on $K$.
\end{lemma}
\begin{proof}
The proof is given in the case that $g$ fixes the ends of $A$, in the other case the proof is analogous.
Let $G$ be a lift of $g$ and let $V_\epsilon=\{(x,y)\in\R\times (0,1): \epsilon<y< 1-\epsilon\}$ be a neighborhood of
$\tilde K=\pi^{-1}(K)$. It is claimed that there exists a homeomorphism $G'$ of $\R\times [0,1]$ satisfying $G'(x,y)=(dx,y)$ in $y=0$ and $y=1$, $G'=G$ on $V_\epsilon$ and $G'(x+1,y)=G'(x,y)+(d,0)$ for every $(x,y)$, where $d$ is the degree of $g$.

To this end, let $R$ be the rectangle $0\leq x\leq 1$, $0\leq y\leq \epsilon$. Note that the above requirements already define $G'$ on the horizontal sides of $R$. Choose a simple arc $s$ in $A$ joining $G'(0,0)$ to $G'(0,\epsilon)$ and disjoint from $G(y=\epsilon)$. Next, define $G'$ on the segment $x=0$, $0\leq y\leq \epsilon$ as a homeomorphism onto $s$. Then define $G'$ on $x=1$, $0\leq y\leq \epsilon$ so as to satisfy the condition $G'(1,y)=G'(0,y)+(d,0)$. Until now, a map $G'$ was defined on the boundary of $R$ and is a homeomorphism from the boundary of $R$ to a simple closed curve $\alpha$. By Jordan-Schoenflies theorem, $G'$ can be extended as a homeomorphism from $R$ to the closure of the bounded component of the complement of $\alpha$. Once $G'$ was defined in $R$ extend it to the whole horizontal strip $0\leq y\leq \epsilon$ so as to satisfy
$G'(x+1,y)=G'(x,y)+(d,0)$.

Repeat the construction in the horizontal strip between $y=1-\epsilon$ and $y=1$.

The map $G'$ obtained is a homeomorphism from the closure of $\tilde A$ onto itself and satisfies the claim. Then project $G'$ to the annulus, giving the required map $g'$.
\end{proof}

For the remainder of this section, we assume that $f$ is a covering map and that $K\subset A$ is a compact subset such that $%
f(K)\subset K$. If
$F:\tilde A\to \tilde A$ is any lift of $f$, $H_F$ is the map given by Lemma
\ref{semi}. Note that $H_F\neq H_{F'}$ if $F$ and $F'$ are different lifts of $f$. If there is no place to confusion, we will write $H$ instead of $%
H_F$.

The proof of Theorem \ref{per} will be divided in two cases.

\subsection{The orientation preserving case.}

\begin{lemma}
\label{fijo} If $f$ preserves orientation, and there exists
$F:\tilde A\to \tilde A$ a lift of $f$ such that $H ^{-1} (0)\neq \emptyset
$, then $\fix(F)\neq \emptyset$ (and so $\fix(f)\neq \emptyset$ ).
\end{lemma}

\begin{proof}
Note that as $f: A \to A$ is a covering, $F:\tilde A\to \tilde A$ is a
homeomorphism. Moreover, $\tilde A$ is homeomorphic to ${\mathbb{R}}^2$ and $F$ preserves orientation because $f$ does. So, by Brouwer's Theorem \ref{brouwer} it is enough to prove that $H ^{-1} (0)$ is a compact $F$-invariant set. Invariance follows from the equality $HF=dH$ (Lemma \ref{semi}, item (2)). To see it is compact, recall from Lemma \ref{semi} that the function $(x,y)\to H(x,y) - x$ defined on $\tilde K$ is bounded. So, we may take $C\in {\mathbb{R}}$ such that $|H(x,y)-x|< C$ on $\tilde K$. Then, $(x,y)\in H ^{-1} (0)$ implies $x\in [-C,C]$,
proving that $H ^{-1} (0)$ is compact.
\end{proof}

\begin{rk}
The fixed point found in the previous lemma does not necessarily belong to $K$ (see Example \ref{e1} in Section \ref{ex}).
\end{rk}

The following remark resembles rotation theory for surface homeomorphisms.

\begin{rk}
\label{rot0} The previous lemma can be restated as follows: if there exists $x\in K$, and a lift $F$ of $f$ such that
$\lim _{n\to \infty} \frac{(F^n(\tilde x))_1}{d^n} = 0$ for a lift $\tilde x$ of $x$, then $\fix(F)\neq\emptyset$ (see Lemma \ref{semi} item 4 and Lemma \ref{fijo}).
\end{rk}

Note, however, that the mere existence of a point $\tilde x \in \tilde A$ such that $\lim _{n\to \infty} \frac{(F^n
(\tilde x))_1}{d^n} = 0$ for some lift $F$ of $f$ does not imply the existence of a fixed point; the set $K$ is key. An example is given in  Example \ref{e2} in Section \ref{ex}.

The following lemma is Lemma 3 in \cite{iprx}.

\begin{lemma}
\label{sur} If $K$ is an essential subset of $A$, then for any lift $F$ of $f
$, the function $H_F: \tilde K \to {\mathbb{R}}$ is surjective.
\end{lemma}

\begin{lemma}\label{preserving}  Let $g: A\to A$ be an orientation preserving covering map of degree $d$, $|d|>1$, and let $K\subset A$ be an essential continuum such that $g(K)\subset K$.  Then,
every lift of $g$
has a fixed point.

\end{lemma}

\begin{proof}
Lemma \ref{sur} states that $H_G$ is surjective for any lift $G$ of $g$.  In particular, for any lift $G$ of $g$, $H_G ^{-1} (0)\neq
\emptyset$. Then, Lemma \ref{fijo} implies that $\fix(G)\neq \emptyset$.

\end{proof}

\subsection{The orientation reversing case.}

\begin{lemma}\label{reversing}  Let $g: A\to A$ be an orientation reversing covering map of degree $d$, $|d|>1$, and let $K\subset A$ be an essential continuum such that $g(K)\subset K$.  Then, every lift of
$g$ to the universal
covering $\tilde A$ has a fixed point.
\end{lemma}

\begin{proof} There are two options:

\noindent {\bf Case 1: $g$ has negative degree and fixes both ends of $A$.}\\
Let $G$ be a lift of $g$ and note that by Lemma \ref{extend}, the map $G$ can be modified without changing its restriction to $\pi^{-1}(K)$, in order to obtain a map that extends to the closure of $\R\times (0,1)$. This extension (also denoted $G$) induces a homeomorphism defined in the compactification with two points $\{-\infty,\infty\}$ of
$\R\times [0,1]$. Note that $G$ carries $-\infty$ to $\infty$ and viceversa. The map $G$ is a homeomorphism of a closed ball, and it can be extended to the whole plane. Theorem \ref{kryst} implies that $G$ has a fixed point in $\pi^{-1}(K)\cup\{-\infty,\infty\}$, as this set is connected (even if $\pi^{-1}(K)$ is not; see Figure \ref{ka}).
Note that this fixed point cannot be $-\infty$ nor $\infty$, as this constitutes a
two periodic orbit. So, every lift of $g$ has a fixed point.  Note, moreover, that in this case, the fixed point belongs to the set $\pi^{-1}(K)$.

\noindent
{\bf Case 2: $g$ has positive degree and interchanges both ends of $A$.}\\
Lemma \ref{extend} can be used to modify the map $g$ in order that it
can be extended to a covering of the closed annulus, and as $g$ reverts the ends, the modification can be performed without changing the fixed point set.

So, by Corollary \ref{periodicos} with $K=\overline A$, we may assume that $g$ is semiconjugate to $m_d$. This in its turn implies that $g$ has an invariant
connector, meaning an inessential continuum of the closed annulus connecting both boundary components (see \cite{iprx}, Corollary 9). Moreover, $g$ has an invariant connector contained in each of the
preimages under
the semiconjugacy of a fixed point of $m_d$. Then any lift $G$ of $g$ must fix one of the lifts of these invariant connectors. Then extend $G$ as a homeomorphism of the whole plane and apply Kuperberg's
Theorem to conclude that $G$ has a fixed point in that invariant connector.\\

\end{proof}

\noindent
{\bf We are now ready to prove Theorem \ref{per}}.

\begin{proof}  By Corollary \ref{final}, it is enough to prove that for all $n$, every lift of $f^n$ has a fixed point. Note that for all $n$, $f^n (K)\subset K$.  If $f$ is orientation preserving,
so is $f^n$ for all $n$, so applying Lemma \ref{preserving} we obtain the result.  If $f$ is orientation reversing, we obtain the result by applying Lemma \ref{preserving} to even powers of $f$,
and Lemma \ref{reversing} to odd powers of $f$.

\end{proof}

\section{Location of periodic orbits}\label{loc}

In this section, we prove Theorem \ref{relleno}.  That is, that the periodic points given by Theorem \ref{per} can be found in $\fil (K)$.
We assume throughout this section that $K$ is an essential continuum such that $f(K)\subset K$. Recall that $f$ is complete because of Theorem \ref{per}.

We will denote $\pi ^{-1}(\fil{K}) = \hat K$. Note that $\hat K$ is not necessarily connected (see Figure \ref{ka}). However, if $K'$ denotes the closure of $\hat K$ in $D$, then $K'$ is  a connected subset of $D$, the two point compactification of $\R\times [0,1]$, as there are no bounded connected components of $\hat K$ and because $K$ is essential.  Also, this set $K'\subset D$ is filled.

\begin{figure}[ht]

\psfrag{k}{$K$} \psfrag{tildea}{$\tilde{A}$} \psfrag{hatk}{$\hat{K}$}
\begin{center}
\subfigure[]{\includegraphics[scale=0.09]{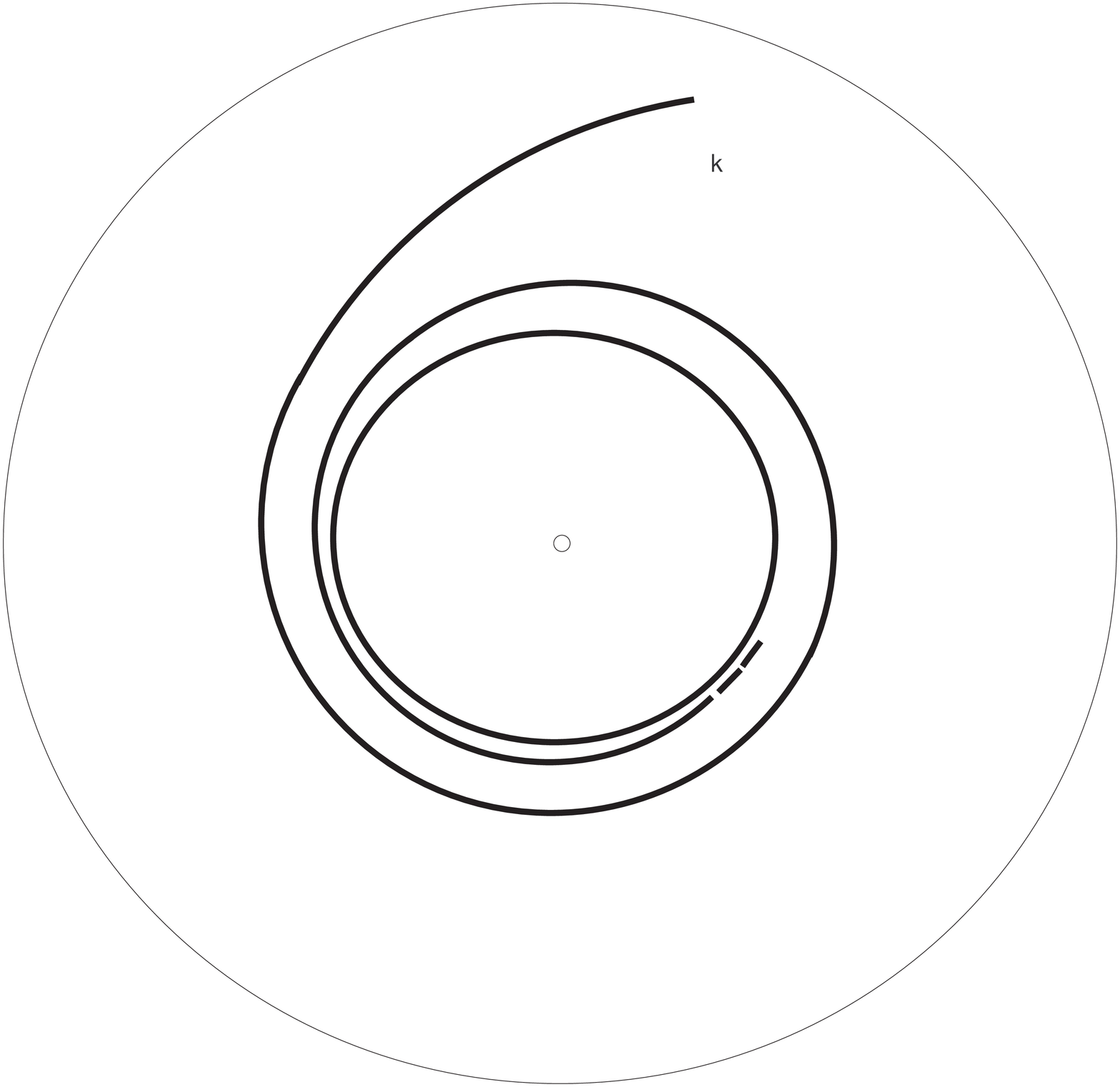}} \subfigure[]{%
\includegraphics[scale=0.22]{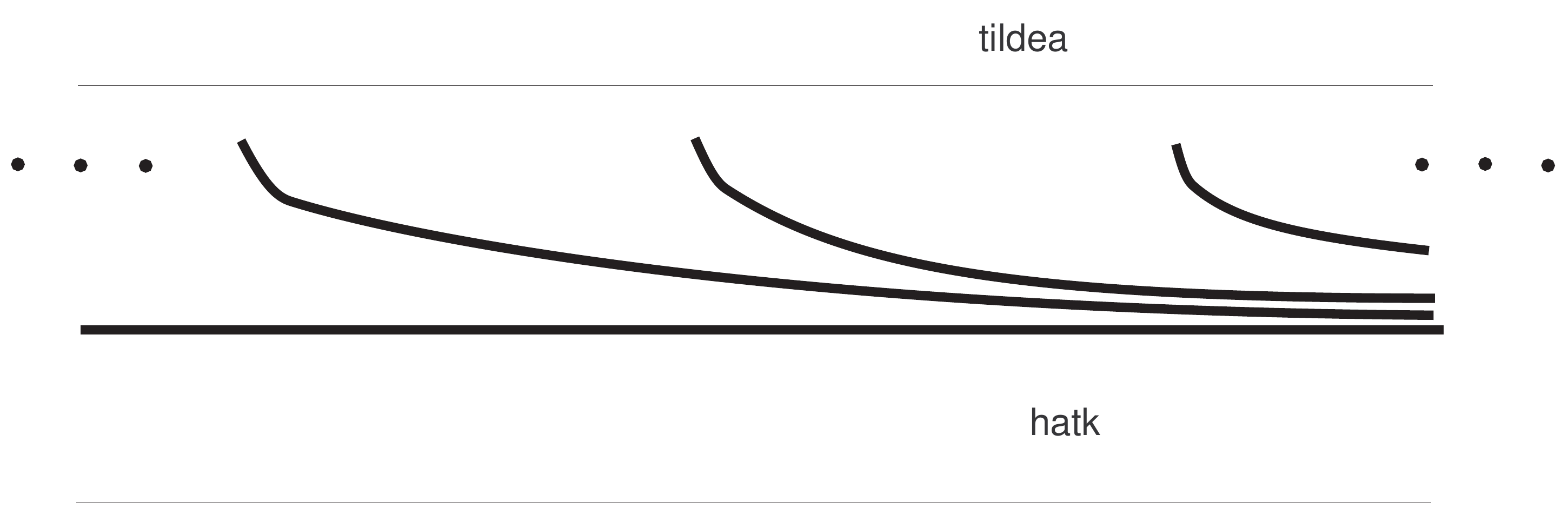}}
\end{center}
\caption{}
\label{ka}
\end{figure}

To prove Theorem \ref{relleno} it is enough to prove that for all $n\in\N$, every lift of $f^{n}$ has a fixed point in $\hat  K$ (see Corollary \ref{final}).

\begin{lemma}  Let $g: A\to A$ be a covering map of degree $d$, $|d|>1$, and let $K\subset A$ be an essential continuum such that $g(K)\subset K$.  Then, every lift of $g$ to the universal
covering $\tilde A$ has a fixed point in $\hat  K$.
\end{lemma}

\begin{proof}  The proof will be subdivided in three cases.

\noindent {\bf Case 1: $g$ is orientation preserving.}

Let $G:\R^{2}\to \R^{2}$ be a lift of  $g$ and let $H_G: \hat K \to \R$ be the map given by Lemma \ref{semi}.  Let $K_0 = H_G^{-1}(0)$ and recall from Lemmas \ref{semi} and \ref{sur} that  $ K_0\neq\emptyset$,
$K_0\subset \hat{K}$, $G(K_0)\subset K_0$ and that $K_0$ is a compact subset of $\R^2$. Suppose for a contradiction that $\fix(G)\cap \hat K=\emptyset$. Let $D$ be the compactification of
$\R\times [0,1]$ with two points $-\infty$ and $+\infty$.  Note that $D$ is a closed disk and that we may assume that $G$ extends to the boundary of $D$, by Lemma \ref{extend}. Moreover, the closure $K'$ of $\hat K$ is  a connected subset of $D$.
Let $P$ be the set of fixed points of $G$ in the interior of $D$.  Note that $P$ does not accumulate at $-\infty$ or $+\infty$.  Define $U$ as the connected component of
$D\setminus P$ containing $K'$, and let $(\tilde U,p)$ be the universal covering of $U$.  Note that $U$ is $G$- invariant, and we claim that there exist a lift
$\tilde G: \tilde U \to \tilde U$ of $G|_U$ having a compact invariant set in  the interior $\tilde U$. To see this, take an open, connected and simply connected neighborhood $V\subset U$ of $K'$ (whose existence is
guaranteed as the set is filled), and note that each
connected component of $p^{-1} (V)$ is mapped homeomorphically onto $V$ by $p$.  Moreover, as $K'$ is connected,
there is only one connected component of $p^{-1} (K')$ in each connected component of $p^{-1} (V)$.  Fix a connected component $ K''$ of
$p^{-1} (K')$ and take the lift $\tilde G$ of $G$ such that $\tilde G ( K'') \subset  K''$.  Note that $p^{-1} (H_G ^{-1}(0))\cap  K''$ is $\tilde G$-invariant, compact and contained in the interior of $\tilde U$. So, as $\tilde G$ is orientation preserving,  Brouwer's Theorem gives a fixed point of $\tilde G$ in the interior of $\tilde U$.  This is a contradiction because by definition, there are no fixed points in the interior of $U$.

\noindent {\bf Case 2: $g$ is orientation reversing, and $d<-1$.}

Note that this has already been proved in Case 1 of the proof of Theorem \ref {per}.

\noindent {\bf Case 3: $g$ is orientation reversing, and $d>1$.}
 Let $U_1$ and $U_2$ be the connected components of $\tilde A\backslash \hat K$.  Note that our hypothesis implies that $G(U_i) \cap U_i = \emptyset, i=1, 2$.  So,  $\fix (G)\subset \hat K$.  It is enough then to prove that $\fix (G)\neq \emptyset$.  This has already been proved in Case 2 of the proof of Theorem \ref {per}.

\end{proof}

 \section{Examples}\label{ex}

In this section we exhibit a series of examples illustrating all the ideas in this article.  Examples \ref{e1}, \ref{e4} and
\ref{nobrou} are particularly interesting, regardless of their connection to the theorems presented in this paper.

\subsection{Location of periodic orbits}\label{e1}
Our first example shows that the periodic points given by Theorem \ref{per}
do not necesarilly belong to $K$.

We will show that there exists a degree two covering map $f$ of the annulus having an essential continuum $K$,
totally invariant, which does not contain fixed points of $f$.\\

\begin{figure}[ht]
\caption{}
\label{figura7}
\begin{center}
\includegraphics[scale=0.2]{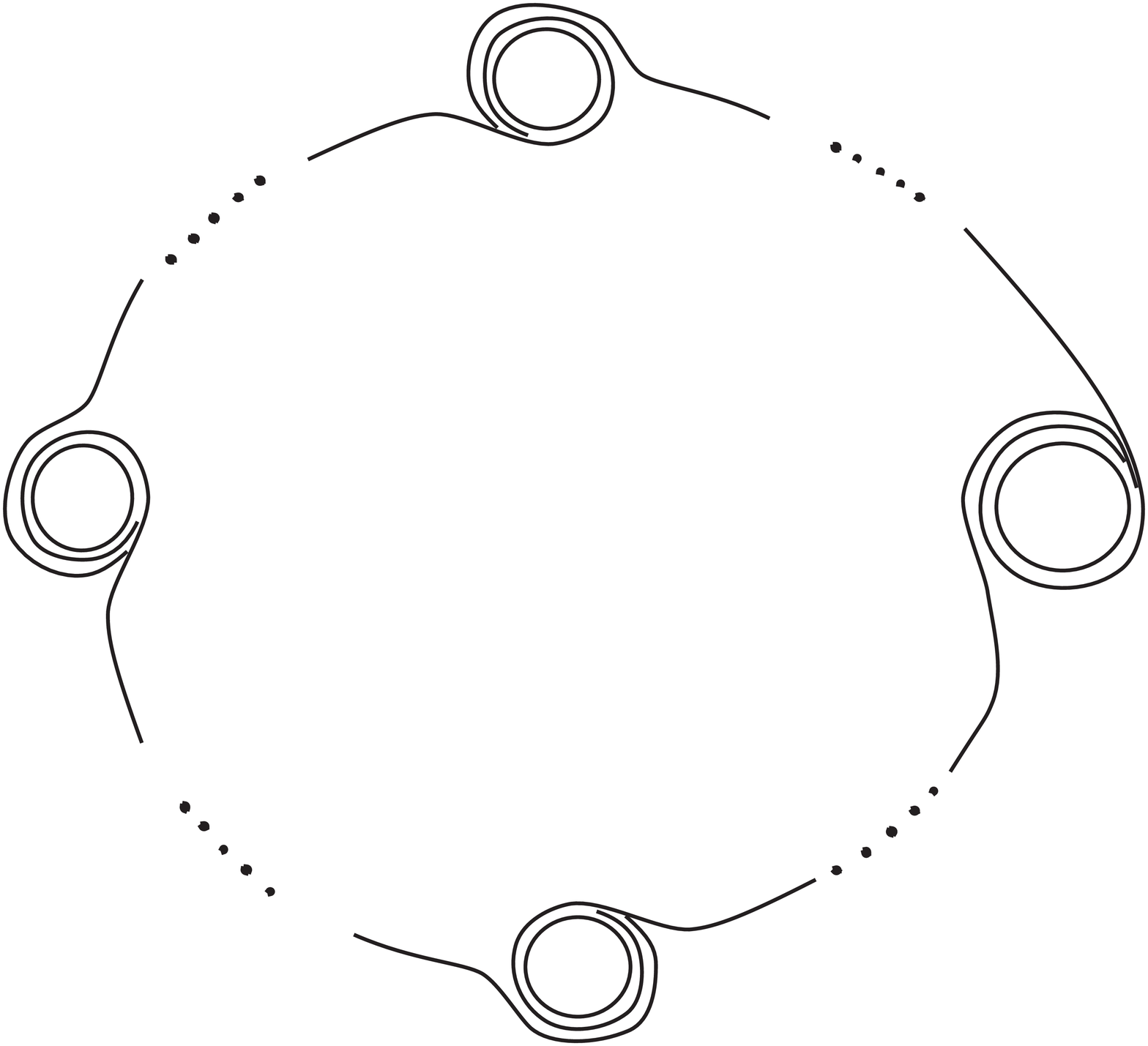}
\end{center}
\end{figure}

We construct an isotopy from $f_0=p_2$, $p_2 (z) = z^2$, to $f_1=f$ in the annulus $A=\C\setminus\{0\}$.
For every $t$, $f_t(z)=f_0(z)$ for every $z$ outside a neighborhood $V$ of the fixed point $1$.
Every $f_t$ will be a homeomorphism from $V$ to $f_0(V)$.
For points in $V$, the restriction of $f_t$ to $V$ will have a unique fixed point at $1$. Around this point $f_t$ performs a Hopf bifurcation (see Figure \ref{figura7}).  That is,
for $t$ close to $0$, the eigenvalues at the fixed point $1$ of $f_t$ have nonzero imaginary part; the modulus is decreasing, and for $t$ equal to $1/2$ the Hopf bifurcation takes place: the modulus of the eigenvalues is equal to $1$, while the imaginary part is different from $0$. Then let $f_t$ for $t>1/2$ be a generic family through the Hopf bifurcation. The following facts hold for $f_1$:  $1$ is an attracting fixed point, there is a repeller simple closed curve $C$ where $f$ is conjugate to a rotation with nonzero rotation number, and every point $z\in V\backslash\{1\}$ has a preorbit in $V$ which converges to $C$.\\

Now let $K$ be the boundary of the basin of $\infty$.
Then $K$ is a totally invariant essential continuum. It is clear that $K$ does not contain a fixed point of $f=f_1$.\\

\subsection{A fixed point free example having a point with zero rotation number}\label{e2}

It may happen that $\lim_{n\to \infty} \frac{(F^n (x))_1}{d^n}=0$ for some lift $F$ of $f$ and $x\in \tilde A$, but $\fix(F)=\fix(f)=\emptyset$.  Just consider a degree $2$ map preserving a
ray of the annulus in which the dynamics is north-south, and lift it preserving a lift  of that ray.

%(You make take, for instance, $f(r,\theta )=(\varphi (r),2\theta )$ where
%$\varphi$ is as in Figure
%\ref{figura2} (b))

\subsection{Changing the lift}\label{e3}

This example shows that the map $f$ may have a lift with fixed points and another lift which is fixed-point free.

Let $f:[0,2\pi]\times (0,1)\to [0,2\pi]\times (0,1)$ con $f(\theta ,r)=(3\theta , \phi(r,\theta ))$, where $\phi$ fixes the rays  $\theta= 0$ and $\theta= \pi$.  On the ray  $\theta= 0$, the
dynamics of $\phi$ is as in Figure \ref{figura2} (a), and on the ray $\theta = \pi$, $\phi$ is as in Figure \ref{figura2} (b).  So, $(0,1/2)$ is fixed by $f$ and you can lift $f$ by fixing
any of the lifts of $(0,1/2)$.  However, if you take a lift $F$ of $f$ fixing any preimage of $\theta= \pi$, then $\fix(F)= \emptyset$.

\begin{figure}[ht]
\psfrag{f}{$\varphi$}
\psfrag{1}{$1$} \psfrag{0}{$0$}
\psfrag{2}{$\frac{1}{2}$}

\begin{center}
\subfigure[]{\includegraphics[scale=0.2]{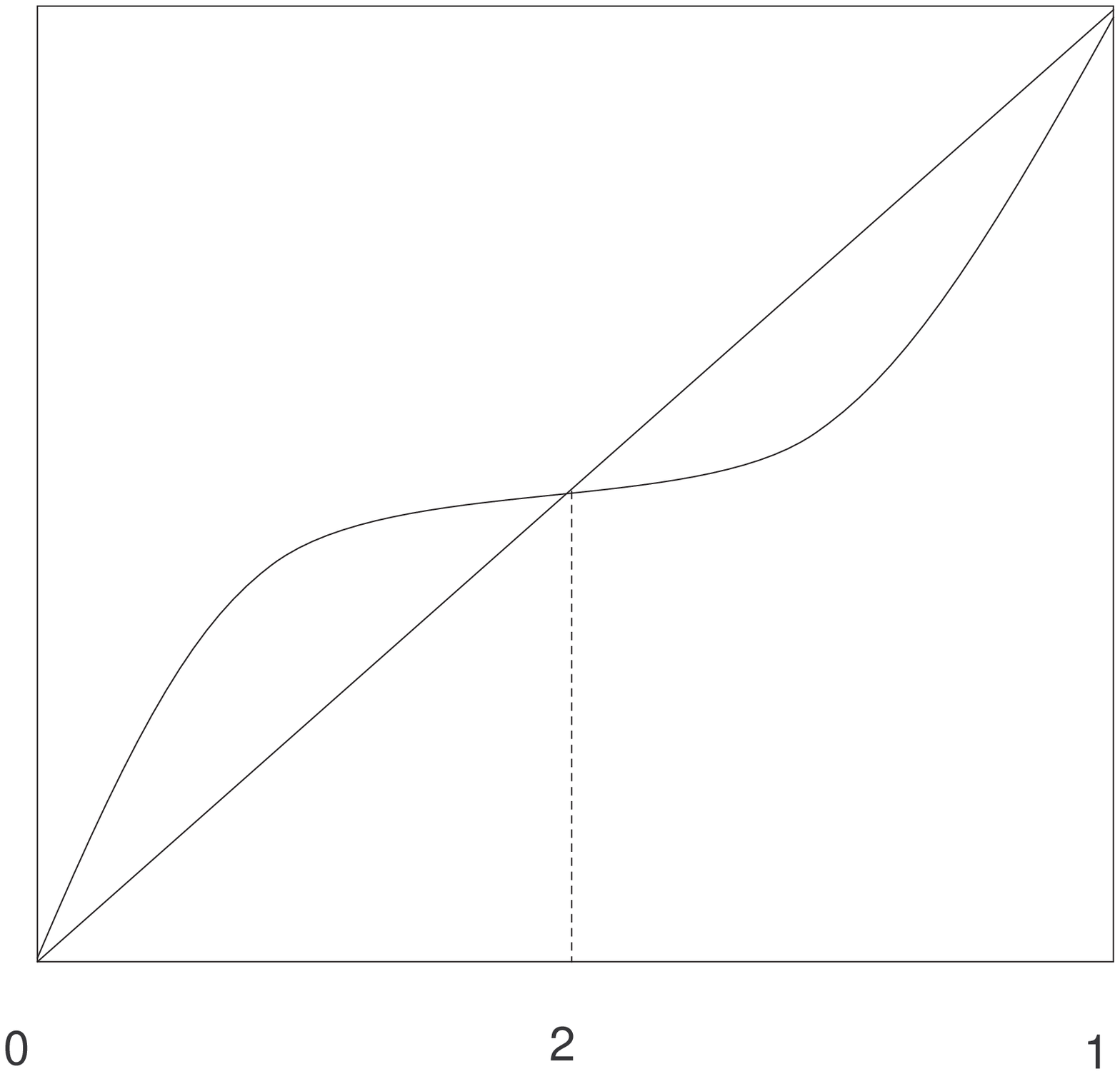}}
\subfigure[]{\includegraphics[scale=0.2]{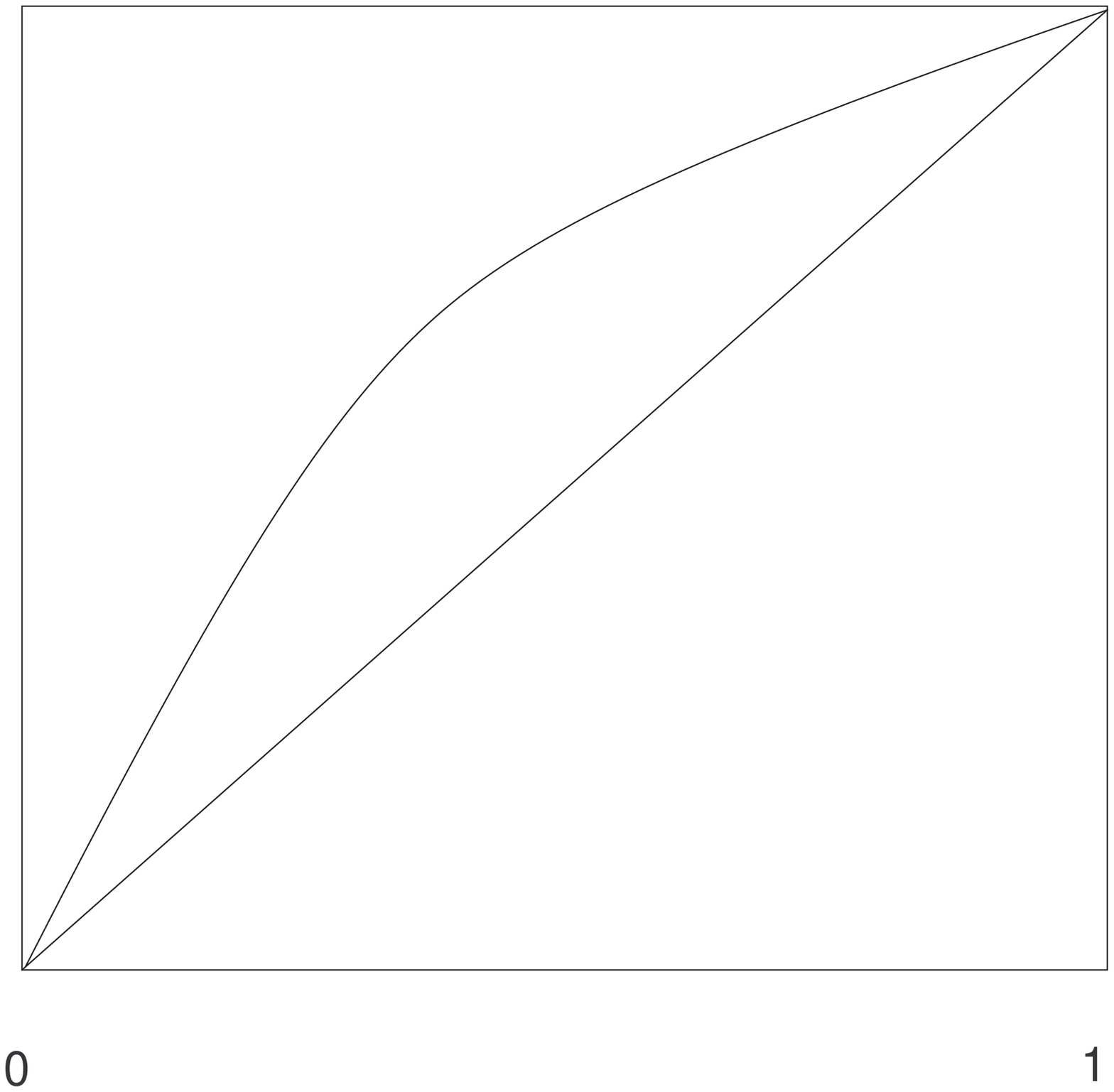}}
\end{center}
\caption{}\label{figura2}
\end{figure}

\subsection{Recurrence and periodic orbits}\label{e4}

As in the fixed-point free degree $2$ covering example $(r,\theta)\mapsto
(2r,2\theta)$ every point is wandering, one may ask if the existence of a
non-wandering point is enough to assure the existence of a fixed point. The
next example shows that this is not the case.

We will construct a degree $2$ covering $f:(0,+\infty)\times S^{1}\to(0,+\infty)\times S^{1}$ such that there is a compact set $K$ satisfying $f(K)= K$ and $\per(f) = \emptyset$.  Of course, $K$ must be inessential and not connected (see Theorem \ref{per} and Lemma \ref{sintasa}).  In fact, in this example $K$ is a Cantor set. We recall
that in \cite{iprx} we showed that for a degree $d>1$ covering $g$ of the circle $\overline{\per (g)}= \Omega (g)$.  This example also shows that this is no longer the case for annulus coverings.

We start with a degree $2$ circle covering having a wandering interval.
Let $g_1: S^{1}\to S^{1}$ be a Denjoy homeomorphism with a wandering interval $I$. Take an open interval  $I_0\subsetneq I$ and an increasing function $h:I\to S^{1}$ such that $h(I_0)=S^1$ and  $h|{_{I\setminus I_{0}}}\equiv g_1$ (see Figure \ref{figura1} (a)). Let $g: S^{1}\to S^{1}$ be the map

\begin{equation*}
g(x)= \left\{
\begin{array}{l}
g_{1}(x) \mbox{ if $x\notin I$} \\
h(x) \ \mbox{ if $x\in I$}%
\end{array}
\right.
\end{equation*}

So, $g$ is a degree $2$ covering of the circle and $g_1(I)$ is a wandering interval for $g$. Besides, if $x_0\in  g_1(I)$ then $K_1 = \omega _g(x_0)$ is a Cantor set and $K_1\cap \per(g) = \emptyset$.

Our example  $f:(0,+\infty)\times S^{1}\to(0,+\infty)\times S^{1}$ has the form $f(r,\theta )=(\phi (r,\theta),g(\theta ))$, where $\phi$ is to be constructed. Let $\psi:S^{1}\to {\mathbb{R}}$, $\psi
(\theta )=dist(\theta, K_1)$ and let $\varphi:(0,+\infty )\to (0,+\infty)$  be as in Figure \ref{figura1} (b).  Define  $\phi  (r,\theta)=\varphi
(r)+r.\psi (\theta)$.

\begin{figure}[ht]

\psfrag{0}{$0$} \psfrag{y=x}{$y=x$} \psfrag{1}{$1$} \psfrag{g}{$g$}
\psfrag{f}{$\varphi$} \psfrag{11}{$\frac{1}{2}$}
\psfrag{i}{$I$}\psfrag{h}{$h$}\psfrag{g1}{$g_1$}
\par
\begin{center}
\subfigure[]{\includegraphics[scale=0.2]{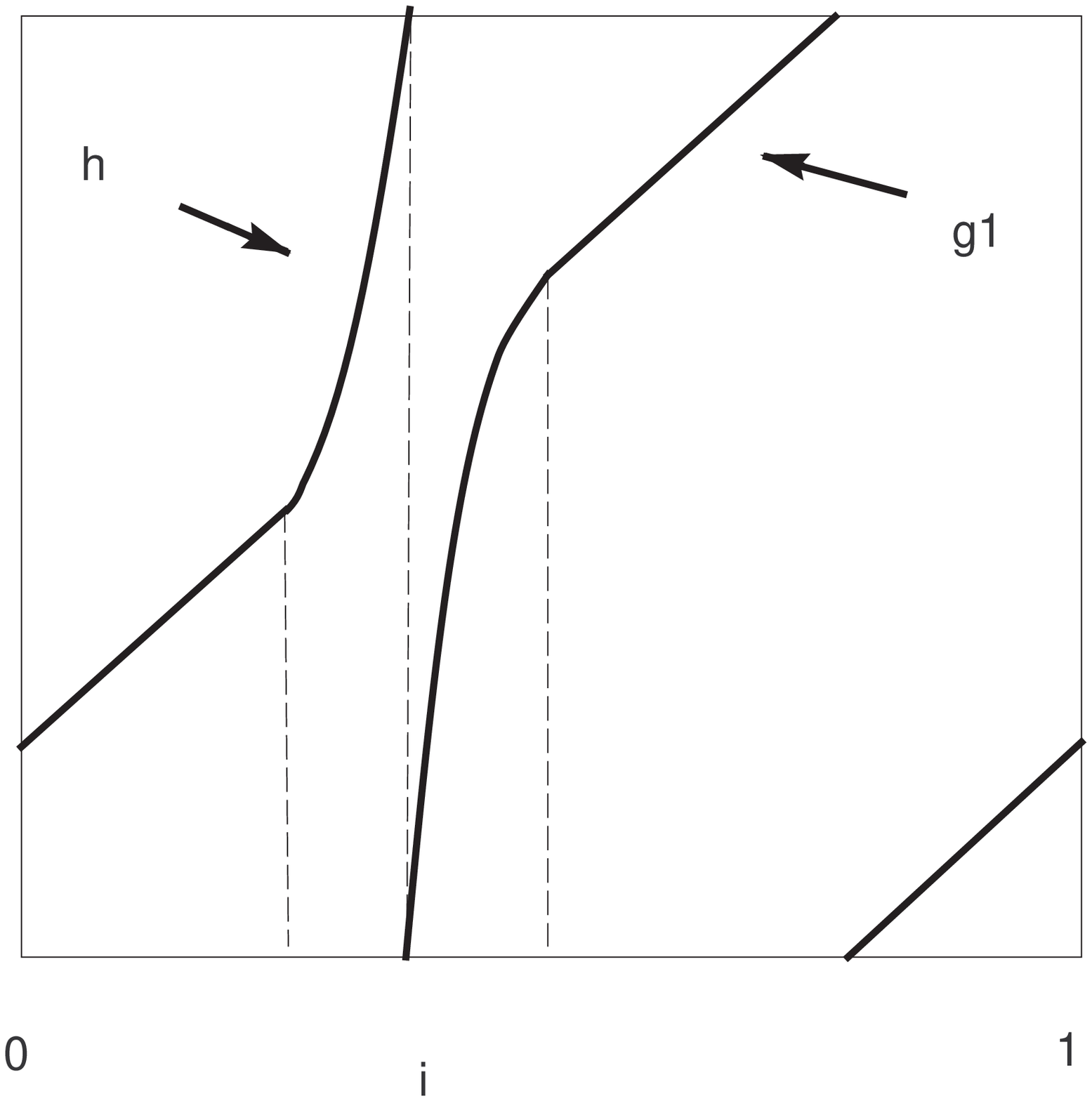}}
\subfigure[]{\includegraphics[scale=0.2]{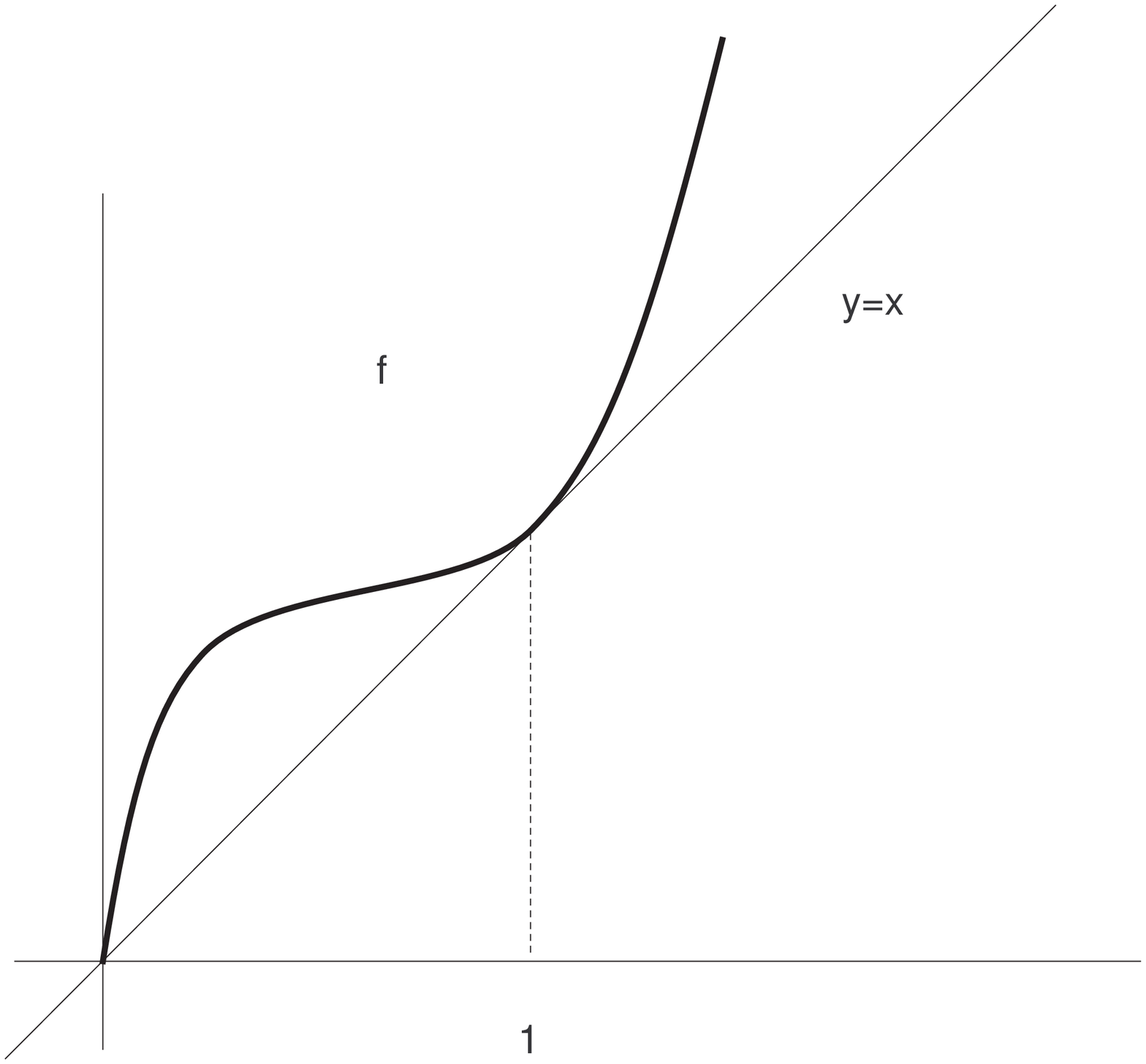}}
\end{center}
\caption{}\label{figura1}
\end{figure}

Note that  $f$ has the following properties :\newline
(1) For fixed $\theta$, let $\phi_{\theta}(r)=\phi(r,\theta )$.  Then,  $\phi_{\theta}$ has fixed points if and only if $\theta\in K_1$, and for $\theta\in K_1$, $\phi_{\theta}$ has a unique fixed point at  $r=1$ .
\newline
(2) $K=\{1\}\times K_1$ is compact and  $f(K)=K$.\newline

Furthermore,  $\per(f) = \emptyset$.  Indeed, if $(r_0,\theta _0)$ is $f$- periodic, then $\theta _0$ must be $g$-periodic.  So, $\theta _0 \notin K_1$.  But this is imposible, as dynamics in the lines $\{(r, \theta): r>0, \theta \notin K_1\}$ is wandering.

Note also that this example can be made $C^1$, if we use the square of the distance in the definition of $\psi$ and enough regularity for the rest of the functions.

\subsection{Non-essential totally invariant subset}\label{e5}  We give an example of a degree $2$ covering of the annulus with a Cantor set $K\subset A$ such that $f^{-1}(K)=K$.  This implies, using Proposition \ref{jp} item 4. in the next section, that $f$ is complete.

Let $g:S^{1}\to S^{1}$ be as in Figure \ref{figura3} (a).
Note that $\Omega (g)=\{0\}\cup K_1$ where $0$ is an attracting fixed point and $K_1$ is an expanding Cantor set with  $g^{-1}(K_1)=K_1$. Let $%
f:(0,1)\times S^{1}\to (0,1)\times S^{1}$, $f(r,\theta )=(\varphi
(r),g(\theta ))$, where $\varphi $ is as in Figure \ref{figura3} (b).
Then,  $K=1/2\times K_1$ is a Cantor set and $f^{-1}(K)=K$.

\begin{figure}[ht]

\psfrag{0}{$0$} \psfrag{y=x}{$y=x$} \psfrag{1}{$1$} \psfrag{g}{$g$} %
\psfrag{gg}{$\varphi$} \psfrag{11}{$\frac{1}{2}$}
\par
\begin{center}
\subfigure[]{\includegraphics[scale=0.2]{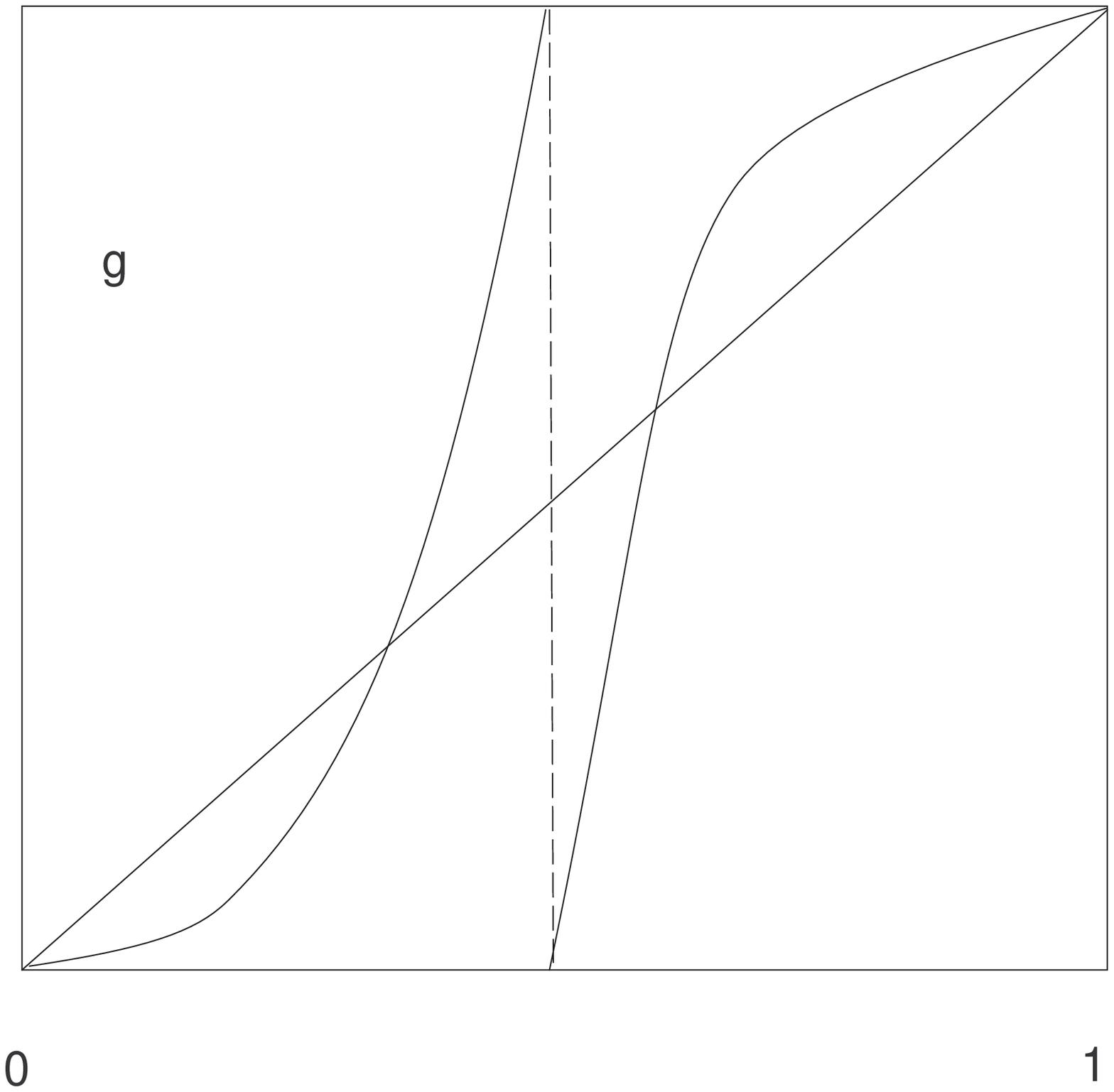}} \subfigure[]{\includegraphics[scale=0.2]{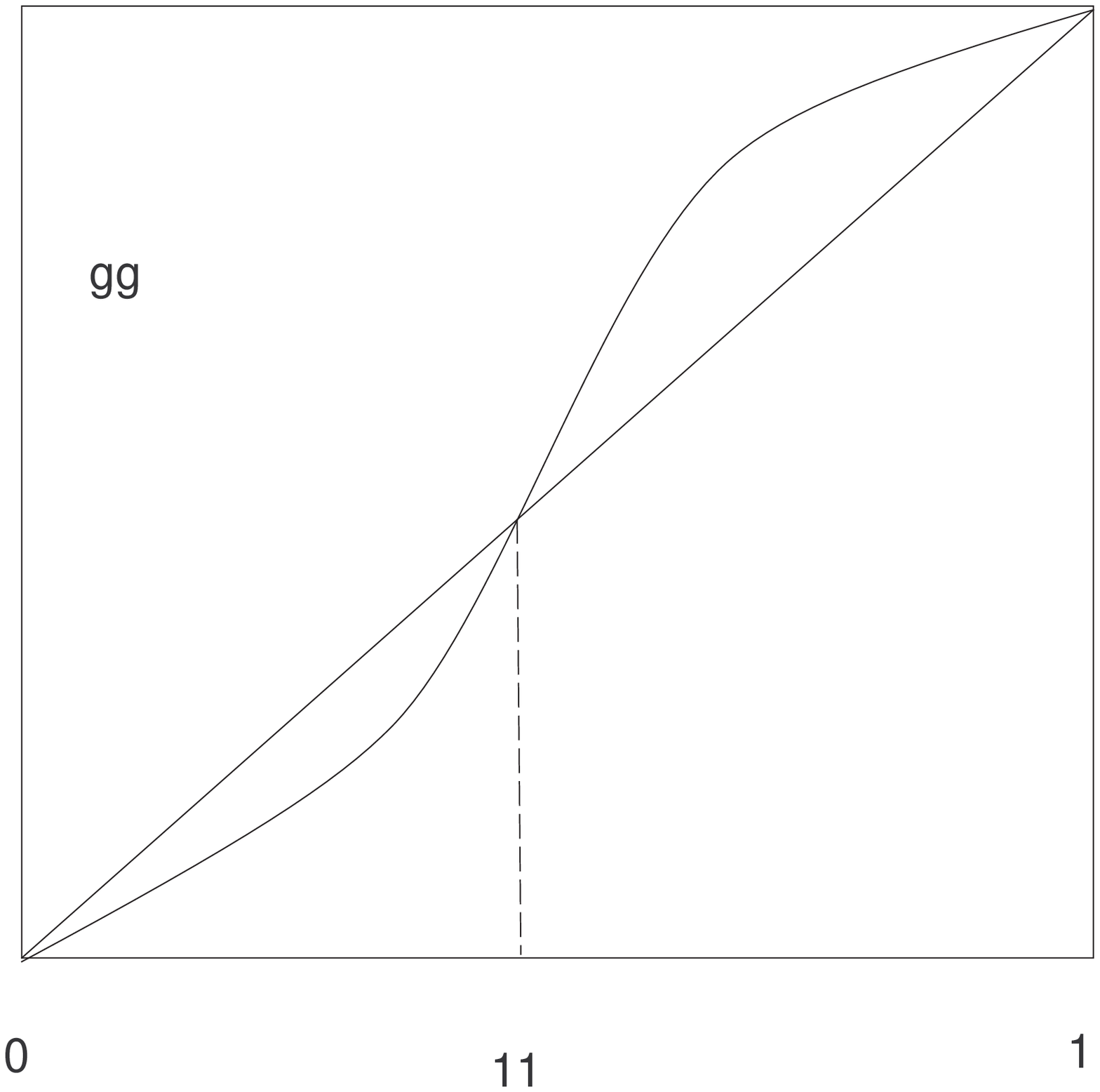}}
\end{center}
\caption{}
\label{figura3}
\end{figure}

\subsection{Fail of Brouwer's Theory}\label{nobrou}

We construct a self- map of the plane $F:\R^{2}\to \R^{2}$ of Brouwer's degree $1$ such that $\Om (F)\neq \emptyset$ and $\per (F) = \emptyset$ (compare with Brouwer's Theorem \ref{brouwer}).

Let $g:S^{1}\to S^{1}$ be the degree $2$ covering map of example \ref{e4} (Figure \ref{figura1} (a)) and let  $K$ be the Cantor set such that $g(K)=K$ and $K\cap \per(g)=\emptyset$. For any degree $d$, $|d|>1$, covering $g:S^1\to S^1$ there exists an increasing semiconjugacy $h_1$ (\cite{iprx} Prop. 1) between $g$ and $q(z)=z^{2}$, that is $h_1g=qh_1$.

 Then,  $h_1(K)$ is compact, $q$- invariant and $\per(q)\cap h_1(K)=\emptyset$. Now, consider the maps $h_2:S^{1}\to [-2,2]$, $h_2(z)=z+\frac{1}{z}$ and $p:[-2,2]\to [-2,2]$, $p(z)=z^{2}-2$. Note that $h_2$ is continuous, surjective and $h_2q=ph_2$. Moreover:
\begin{itemize}
\item $K_1=h_2(h_1(K))$ is compact and $p(K_1)\subset K_1$.
\item $\per (p)\cap K_1=\emptyset$.
\end{itemize}

We need one more auxiliary function to make our example $F:\R^{2}\to \R^{2}$  have degree $1$. Let  $f:\R\to \R$ be as in Figure \ref{pepito}(a), so that $f|_{[-2,2]}=z^{2}-2$. Now we proceed in the same fashion than in example \ref{e4}. Let $\psi:\R\to {\mathbb{R}}$, $\psi
(x)=dist(x, K_1)$ and let $\varphi:\R\to \R$  be as in Figure \ref{pepito} (b).

\begin{figure}[ht]

\psfrag{-2}{$-2$} \psfrag{y=x}{$y=x$} \psfrag{2}{$2$} \psfrag{g}{$g$} %
\psfrag{f}{$f$} \psfrag{e}{$\varphi_{x} \mbox{with } x\notin K$}
\psfrag{ee}{$\varphi_{x} \mbox{with } x\in K$}
\psfrag{i}{$I$}\psfrag{h}{$\varphi$}\psfrag{g1}{$g_1$}
\par
\begin{center}
\subfigure[]{\includegraphics[scale=0.2]{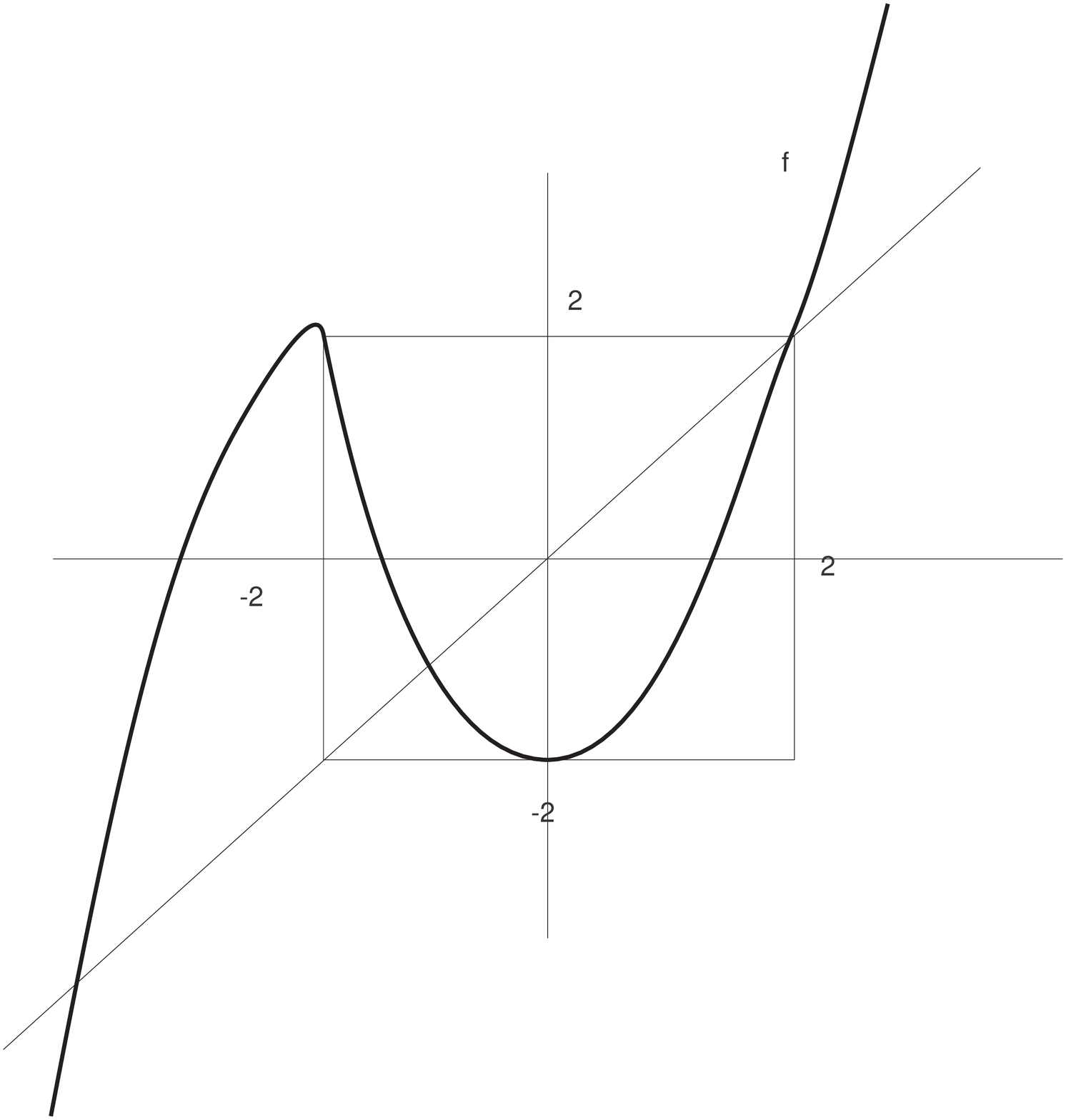}} \subfigure[]{%
\includegraphics[scale=0.2]{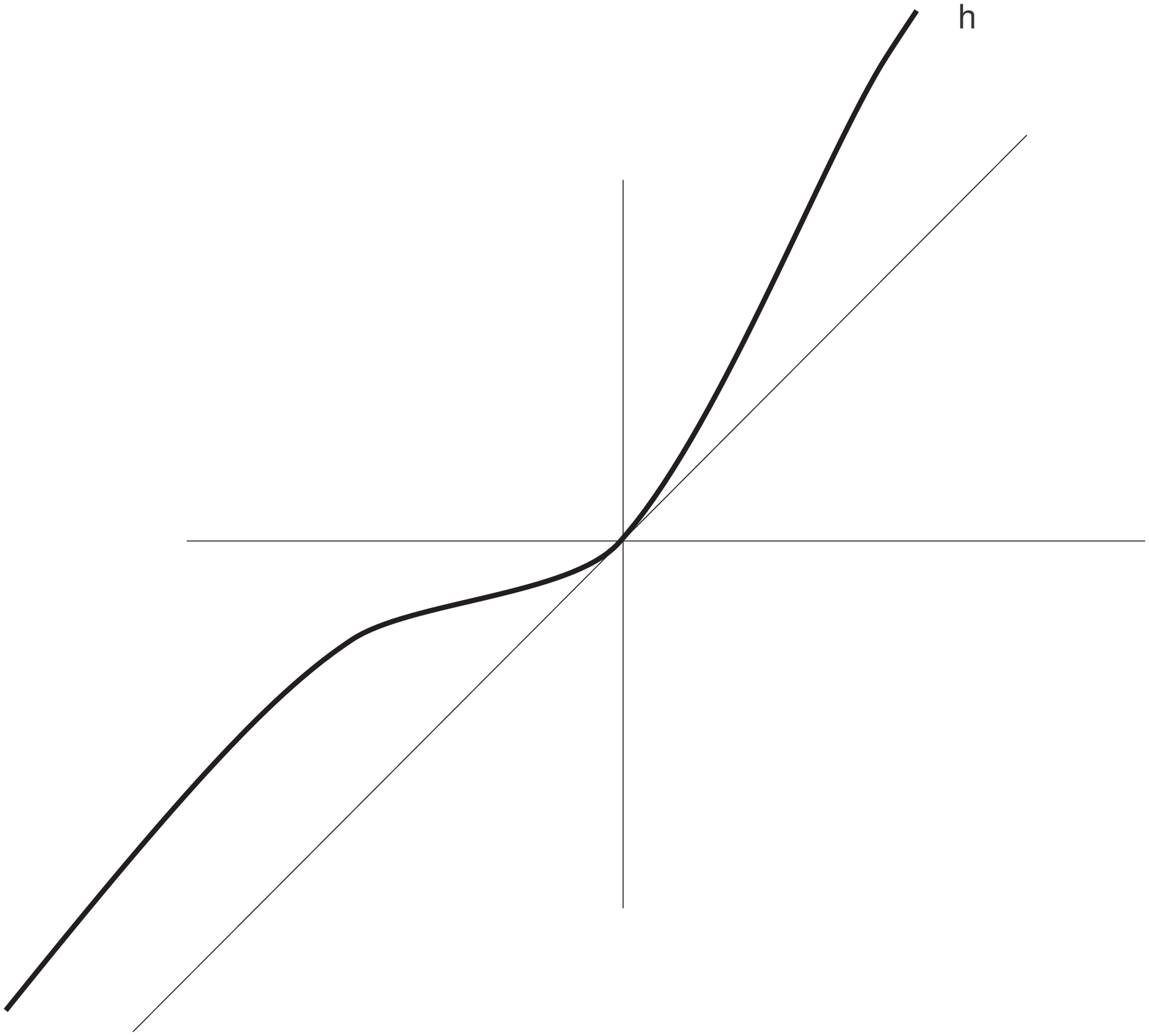}}
\end{center}
\caption{}
\label{pepito}
\end{figure}

 Define  $\phi  (x,y)=\varphi
(y)+\psi (x)$ and $F:\R^{2}\to\R^{2}$ such that $$F(x,y)=(f(x), \phi (x,y)).$$ Note that $K_2=K_1\times \{0\}$ is compact and $F$-invariant, so
$\Omega (F)\neq\emptyset$. However, $\per (f) = \emptyset$ as the lines $\{(x,y): x \notin K_1\}$ have wandering dynamics, and no line $\{(x,y): x \in K_1\}$ is periodic.

\subsection{Another example without periodic points}

There are essentially two examples of covering maps of the annulus without periodic points:
The first one, given in the introduction, is conjugate to $p_d(z)=z^d$ acting in the punctured unit disc.
The second one was given in example 5.4, in this case the map has nonempty nonwandering set.
As in example 5.3, examples of covering maps with any finite number of periodic points can be constructed,
the question is if there exist examples of covering maps which are not complete but satisfy the growth rate inequality for periodic points.
Note also that if both ends are attracting or both repelling, then the map is complete. In all the examples
of non-complete maps it holds that one end is attracting and the other repelling, but this is not necessary.
For example,
let $f:(z,x)\in S^1\times (0,1)\to (z^3, \varphi_z(x))$, where $\varphi_z$ is an increasing homeomorphism of
the interval $(0,1)$ for each $z$, such that $\varphi_1^n(x)\to 1$ and $\varphi^n_{-1}(x)\to 0$ for every $x$.
This implies that the ends are neither attracting nor repelling. It will be shown here that the a map like $f$ can be constructed without any periodic point. This example was communicated to us by the referee.\\
Let $A=S^1\times \R$ and assume that the map is given by $f(z,x)=(z^3,x+t_z)$ where $t_z$ varies continuously with $z$ and the set of numbers $\{t_p\}$ with $p$ periodic of $m_3$ is rationally independent. This means that $f$ has no periodic points, because if $\{p_1,\ldots,p_n\}$ is a periodic orbit of $z^3$, then $f^n(p_1,x)=(p_1,x+\sum_i t_{p_i})$, but by assumption $\sum_i t_{p_i}\neq 0$. To construct the function $t_z$ choose a rationally independent sequence of positive numbers $\{a_i:i\geq 0\}$ and enumerate the periodic points of $m_3$ as $\{p_n:n\geq 0\}$, being $p_0=1$, $p_1=-1$. Define by induction a sequence of functions $z\to t^n_z$ begining with $t^0_z$ as any continuous function such that
$t^0_{p_0}=a_0$ and $t^0_{p_1}=0$. Given $n>0$ define $t^{n}_z$ as follows: $t^{n}_z=t^{n-1}_{z}$ outside a neighborhood of $p_n$ not containing any $p_i$
for $i<n$, $0\leq t^{n-1}_z-t^n_z<2^{-n}$ and $t^n_{p_n}\in a_n\Q$. The sequence of functions $z\to t^n_z$ converges uniformily to a function $z\to t_z$ satisfying the required properties.

\section{Applications}\label{app}

We devote this section to applications of Theorem \ref{per} and Lemma \ref{fijo} to dynamics. Throughout this section, $f: A \to A$ is a degree $d$, $|d|>1$,
covering.

By \textit{attracting set} we mean a proper open subset $U$ such that $\overline{f(U)}\subset U$.

A subset $X\subset A$ is \textit{totally invariant} if $f^{-1}(X) = X$.

\begin{prop}\label{jp} Any of the following hypothesis imply that $f$ is complete.

\begin{enumerate}

\item \label{1} There is an essential attracting set.

\item \label{2} Each end of $A$ is attracting.

\item \label{3} $f$ extends to a map of the two-point compactification of $A$
in such a way that it is $C^1$ at the poles.

\item \label{4} $f$ preserves orientation and there exists a compact totally invariant (not necessarily connected) subset.

\item \label{5} $f$ preserves orientation and there exists an invariant continuum $K$ such that $h(K)$ is not reduced to a point ($h$ was defined on Corollary \ref{periodicos}).
\end{enumerate}
\end{prop}

\begin{proof}
\ref{1}: Let $U\subset A$ be an essential open set such that $\overline{f(U)}%
\subset U$. Then, $K=\cap_{n\geq 0} f^n (U)$ is an invariant continuum.
Besides, it is essential, because $f^n (U)$ is essential for each $n$ as $|d|>1$. The result now follows from Theorem \ref{per}.

\ref{2}: If both ends are attracting, then the complement of both basins of attraction  is an  essential invariant continuum, and we
may apply Theorem \ref{per}.

\ref{3}: Note that in this case both ends must be attracting, as the
derivative in the compactification must be $0$ at the poles.  Indeed, note that $f$ is a $d:1$  branched covering in a neighbourhood of the
pole. So the winding number $I_{f\gamma}(p)=d$ whenever $\gamma$ is a small circle with center $p$. On the other hand, if
a map $g$ has a fixed point at $p$ isolate in $g^{-1}(p)$ and nonvanishing differential at $p$, then $|I_{g\gamma}(p)|\leq 1$ for a small curve $\gamma$ such that $I_\gamma(p)=1$.

\ref{4}: Let $X$ be a compact set such that $f^{-1}(X) = X$. It is enough to show that $m_d^{-1}(h(X))= h(X)$. Indeed, if $m_d^{-1}(h(X))= h(X)$, then  $h(X)$ is dense in $S^1$.  As $h(X)$ is also compact, $h(X) = S^1$.  So, $H_F$ (from Lemma \ref{semi}) is surjective for any lift $F$ of
$f$.  Using the same argument, $H_G$ is surjective for any lift $G$ of $f^n$, and we are done by Lemma \ref{fijo} and Corollary \ref{final}.

To prove that $m_d^{-1}(h(X))= h(X)$ we first claim that if $x\neq y$ and $f(x)= f(y)$, then $h(x)\neq h(y)$. Let $F$ a lift of $f$ and $\tilde x$ a lift of $x$.  Define $\tilde y_j = F^{-1}(\tilde x +j), j=0, \ldots, |d|-1$.  Then,
$\pi(\cup_{j=0}^{|d|-1} \tilde y_j )= f^{-1}(x)$.  Moreover, $dH(\tilde y_j)= HF(\tilde y_j)= H(\tilde x+j)= H (\tilde x)+j$, and so $H(\tilde y_j)= \frac{H(\tilde x)}{d}+\frac{j}{d}$. This proves the claim, because $\pi H(\tilde y_j) = h\pi (\tilde y_j)$ and $\{\pi H(\tilde y_j): 0\leq j\leq |d|-1\}$ has $|d|$ different elements because of the computation we just did.

It is obvious that $h(X)\subset m_d^{-1}(h(X))$ by the semiconjugacy equation.  To prove that $m_d^{-1}(h(X)) \subset h(X)$, let $z\in h(X)$ and we will prove that
$m_d^{-1}(z) \subset h(X)$.  As $z\in h(X), z = h(x), x \in X$.  Let $f^{-1}(x) = \{x_1,\ldots , x_d\}$.  By hypothesis $x_i\in X$ for all $i= 1 , \ldots, d$. Now,
$$m_dh(x_i)= hf(x_i)= h(x) = z.$$   \noindent So, $h(x_1), \ldots, h (x_d) \in m_d^{-1}(z)$.  As $m_d^{-1}(z)$ has exactly $d$ elements, and $h(x_1), \ldots, h (x_d)$
are exactly $d$ elements by the claim, we have $m_d^{-1} (z) = \{h(x_1), \ldots, h (x_d)\}$.

\ref{5}:  Note that $h(K)$ is an invariant interval that is not reduced to a point, and so $h(K) = S^1$, which implies that $H_F$ is surjective for any lift $F$ of $f$, and we conclude as in the
previous item.
\end{proof}

The following application shows how the existence of a periodic orbit can imply existence of infinitely many of them.  The proof is immediate from item (3) in the previous proposition.

\begin{clly}  Let $f: S^2 \to S^2$ be a $C^1$ degree $d$ map, $|d|>1$, and ${p,q}$ a two-periodic totally invariant orbit ($f^{-1}(\{p,q\}) = \{p,q\}, f(p)=q, f(q)=p$). If
$f:S^2\backslash \{p,q\}\to S^2\backslash \{p,q\}$ is a covering, then $f$ has periodic points of arbitrarily large period.

\end{clly}

We will make some calculations that will be used in the following lemma.

Fix a lift $F= F_0$ of $f$, and for any $k\in {\mathbb{Z}}$, define the maps $F_k
(x) = F(x) + k$. Then, for any $m\in {\mathbb{N}}$ and $x\in \tilde A$,

\begin{equation}  \label{eq1}
F_k^m (x) = F^m (x) + \sum _{i=0}^{m-1} k d ^i = F ^m (x) + \frac{k(1-d^m)}{%
1-d}.
\end{equation}
The computation is straightforward, following from the fact that for any $%
k\in {\mathbb{Z}}$ and $x\in \tilde A$, $F(x+k) = F(x)+ dk$.

\begin{lemma}\label{rotk} Suppose that there exists a compact set $K\subset A$ such that $f(K)\subset K$.
If there exists $x\in K$, and a lift $F$ of $f$ such that $\lim _{m\to
\infty} \frac{(F^m (\tilde x))_1}{d^m} = \frac{k}{d^n-1}$, $k\in {\mathbb{Z}}%
, n\geq1$, for a lift $\tilde x$ of $x$, then there exists $z\in \tilde A$
such that $F^n (z) = z+k$. In particular, $\per(f)\neq \emptyset$.
\end{lemma}

\begin{proof}
We have to show that the map $F^n -k$ has a fixed point. By Remark \ref{rot0}%
, it is enough to show that $\lim _{m\to \infty} \frac{(G^m (\tilde x))_1}{%
(d^n)^m} = 0$, where $G= F^n -k$ (see Remark \ref{rot0} and note that $G$ is
a lift of $f^n$).

\begin{equation*}
\lim _{m\to \infty} \frac{(G^m (\tilde x))_1}{(d^n)^m} = \lim _{m\to \infty}
\frac{((F^n -k)^m (\tilde x))_1}{d^{nm}}  \overset{%
\mbox{\scriptsize
\begin{tabular}{c}
(\ref{eq1})
\end{tabular}}}{= }\lim _{m\to \infty} \frac{(F^{nm} (\tilde x))_1- \sum
_{i=0}^{m-1} k d ^{ni}}{d^{nm}}=
\end{equation*}
\begin{equation*}
\frac{k}{d^n-1}- \lim _{m\to \infty} \frac{k}{d^{nm}} \sum _{i=0}^{m-1} d
^{ni}.
\end{equation*}%
\noindent Now,
\begin{equation*}
\lim _{m\to \infty} \frac{k}{d^{nm}} \sum _{i=0}^{m-1} d ^{ni} = \lim _{m\to
\infty} \frac{k}{d^{nm}} \frac{1-d^{nm}}{1-d^n} = \frac{k}{d^n -1}.
\end{equation*}

\end{proof}

\begin{prop}\label{appper} Any of the following hypothesis imply that $f$ has periodic points.
\begin{enumerate}

\item \label{i1} There exists $x\in A$ with bounded forward orbit such that $\lim_{m\to
\infty} \frac{(F^m (\tilde x))_1}{d^m}= \frac{k}{d^n-1}$, for some lift $\tilde x
$ of $x$, $k\in {\mathbb{Z}}, n\geq 1$.

\item \label{i2} There exists an invariant continuum.
\end{enumerate}
\end{prop}

\begin{proof}\ref{i1}:  Let $K$ be the closure of the forward orbit of $x$.  Then, $K$ is compact, $f(K)\subset K$, and $\lim_{m\to
\infty} \frac{(F^m (\tilde x))_1}{d^m}= \frac{k}{d^n-1}$, for a lift $\tilde x
$ of $x$, $k\in {\mathbb{Z}}, n\geq 1$.  Then, by Lemma \ref{rotk} $\per(f)\neq \emptyset$.

\ref{i2}:  If the continuum happens to be essential, then $f$ is complete by Theorem \ref{per}. Otherwise, if $f$ preserves orientation, this follows from Lemma \ref{sintasa}.  If $f$ reverses
orientation, one applies Kuperberg's Theorem \cite{krys}.
\end{proof}

We finish this paper with an open question:  if the covering assumption is dropped and we remain with a degree $d$ map of the annulus ($|d|>1$). If $K$
is an invariant essential continuum, is $f$ necessarily complete?

\end{document}